\documentclass[a4paper,10pt, twocolumn]{article}

\usepackage{microtype}
\usepackage[utf8]{inputenc}
\usepackage[cmex10]{amsmath}
\usepackage{amssymb, latexsym, amsthm}
\usepackage{mathtools}
\usepackage{authblk}

\usepackage{appendix}

\usepackage{color}    
\usepackage{stmaryrd}
\usepackage{enumitem}
\usepackage{fullpage}

\usepackage{tikz}     
\usepackage{pgfplots} 
\pgfplotsset{width=\columnwidth,height=7cm}

\usepackage{xspace} 

\newtheorem{theorem}{Theorem}
\newtheorem{lemma}[theorem]{Lemma}
\newtheorem{proposition}[theorem]{Proposition}
\newtheorem{corollary}[theorem]{Corollary}
\theoremstyle{definition}
\newtheorem{definition}[theorem]{Definition}
\newtheorem{orexample}{Example}

\newcommand*{\quotes}[1]{``#1''}
\newcommand*{\procedure}[1]{\texttt{{\sc #1}}}

\newcommand*{\Ordo}{O}
\newcommand*{\abs}[1]{{\lvert #1\rvert}}

\newcommand*{\union}{\bigcup}

\newcommand*{\set}[1]{\{\,#1\,\}}
\newcommand*{\defined}{\coloneqq}

\newcommand{\ind}[1]{\llbracket #1 \rrbracket}

\renewcommand{\P}{{\rm P}}              
\newcommand{\lb}{{\log_2}}              
\newcommand{\length}{{\ell}}

\newcommand{\SetB}{\mathbb{B}}          
\newcommand{\SetN}{\mathbb{N}}          

\newcommand*{\splus}{\!\!+\!\!}
\newcommand*{\sminus}{\!\!-\!\!}

\newcommand*{\sspace}{\ensuremath{X}\xspace}        
\newcommand*{\range}{\ensuremath{Y}\xspace}         
\newcommand*{\probcon}{{\sspace,\range}\xspace}     
\newcommand*{\cont}{{\sspace\range}}                  
\newcommand*{\sspaceset}{{\cal X}}                  
\newcommand*{\rangeset}{{\cal Y}}                   
\newcommand*{\tset}{{\ensuremath{\cal T}}\xspace}   
\newcommand*{\emptrace}{\langle\rangle}             
\newcommand*{\rset}{{\cal R}}                       

\newcommand*{\m}{{\ensuremath{\bf m}}\xspace}       
\newcommand*{\mf}{{\m_{\sspace\range}}}               
\newcommand*{\cmf}[2]{\cond{\mf{}}{#1}{#2}}         

\newcommand*{\perfcon}{M^\P_{\sspace\range}}        
\newcommand*{\mptm}{\ensuremath{M_{{\rm ot}}}\xspace}  
\newcommand*{\mptmcon}{M_{{\rm ot},\cont}}     

\newcommand*{\niah}{{\rm NIAH}}                                  
\newcommand*{\niahp}{\ensuremath{u_{\niah,\sspace\range}}\xspace} 
\newcommand*{\fK}[1]{\cK{#1}{\sspace,\range}}                    
\newcommand*{\pK}[1]{\cK{#1}{\sspace,\range}}                    

\newcommand*{\uniah}{u_{\niah}}

\newcommand*{\niahc}{\niah_{\sspace\range}}

\newcommand*{\fbad}{f_{{\rm bad}}}

\newcommand*{\cond}[4][]
{\mbox{\ensuremath{{#2}_{#1}(\mspace{1mu}#3\mspace{2mu}\vert\mspace{2mu}#4\mspace{1mu})}}}
\newcommand*{\cK}[3][]{\cond[#1]{K}{#2}{#3}}

\newcommand\copyrighttext{%
  \footnotesize
Published as: pp.\ 167--174, in
\emph{Proceedings of 2014 IEEE Congress on Evolutionary Computation (CEC)},
July 6-11, 2014, Beijing, China.
DOI: 10.1109/CEC.2014.6900546 \copyright 2014 IEEE}
\newcommand\copyrightnotice{%
\begin{tikzpicture}[remember picture,overlay]
\node[anchor=south,yshift=10pt] at (current page.south) {\fbox{\parbox{\dimexpr\textwidth-\fboxsep-\fboxrule\relax}{\copyrighttext}}};
\end{tikzpicture}%
}

\begin{document}

\date{}
\title{Free Lunch for Optimisation under the Universal Distribution}

\author[1]{Tom Everitt}
\author[2]{Tor Lattimore}
\author[3]{Marcus Hutter}

\affil[1]{Stockholm University, Sweden. Email: everitt@math.su.se}
\affil[2]{University of Alberta, Edmonton, Canada}
\affil[3]{Australian National University, Canberra, Australia}

\maketitle
\thispagestyle{empty}
\begin{abstract}
Function optimisation is a major challenge in computer science. The No Free Lunch theorems state that if all functions with the same histogram are assumed to be equally probable then no algorithm outperforms any other in expectation. We argue against the uniform assumption and suggest a universal prior exists for which there is a free lunch, but where no particular class of functions is favoured over another. We also prove upper and lower bounds on the size of the free lunch.
\end{abstract}
\setcounter{tocdepth}{1}

\vspace{.5em}

\copyrightnotice

\section{Introduction}

Finite black-box optimisation is the problem of finding an optimal value
(usually the maximum or minimum) of a target function
$f\colon\,X\to Y$ where $X$ and $Y$ are
finite. A wide range of tasks may be formulated in this setting.
For instance, drug-design may be viewed as
the task of finding a mix of chemicals that maximises recovery chances.
Since experimentation is expensive it is crucial that the best drug be found
as soon as possible.

It is desirable to find optimisation algorithms that perform well on
a wide variety of target functions, as this minimises the need for
fine-tuning the algorithm to the problem.
Indeed, several such algorithms exist and are regularly employed
in practice;
examples include hill-climbing and simulated annealing, as well
as genetic algorithms.
However, the theoretical understanding of the conditions permitting
such ``universal'' algorithms remains limited
\cite{Christensen2001,Streeter2003,Whitley2006,Jiang2011}.
To approach this problem, we derive bounds for expected
optimisation-performance under assumptions justified in all
(or virtually all) optimisation settings.

The original No Free Lunch (NFL) theorems state that when the global performance
of an optimisation algorithm is measured by taking a uniform average of its
performance over all functions from $X$ to $Y$, then no algorithm is
better than random  (assuming no point is sampled more than once)
\cite{Wolpert1997}.
The uniform assumption is justified by assuming the
absence of prior knowledge and the results are often used to claim that
no optimisation algorithm can be universal.

There is, however, another viewpoint. If we assume that the function
$f\colon\, X\to Y$ to be optimised is generated by some (unknown) computer
program,
then taking a uniform prior over programs is arguably more natural.
This is a reasonable assumption based on the commonly held view
that the universe is likely to be (stochastically) computable
\cite{Fredkin1992, Wolfram2002, Hutter2012}.
The distribution on functions induced by this approach is the famous universal
lower-semicomputable semi-distribution\footnote{
The use of lower semicomputable semi-distributions rather than regular
computable distributions is technical only and may be ignored by the reader
unfamiliar with algorithmic information theory.}
developed by Solomonoff and others \cite{Li2008}.
The universal distribution satisfies many nice properties,
both theoretical and philosophical. It is a natural choice when formalising Occam's razor in combination with
Epicurus' principle of multiple explanations since it favours simplicity over complexity without disregarding
the possibility that the truth is complex \cite{Hutter2005,Rathmanner2011}.
The universal distribution also exhibits a range of other desirable properties,
discussed further in Section \ref{sec:unidist}.

If performance is measured in expectation with respect to the universal
distribution, then the no free lunch theorems can no longer be applied.
Indeed, under some highly technical
conditions Streeter \cite{Streeter2003} showed that there is a free lunch
for optimisation under Solomonoff's universal distribution.
Tightly related to the universal distribution is
Kolmogorov complexity: Borenstein and Poli \cite{Borenstein2006}
discuss Kolmogorov complexity and optimisation, and also give a good
account of previous research in this area (see also \cite{McGregor2006}).
Several authors report
on Kolmogorov complexity not being perfectly related to searchability
\cite{Schumacher2001,Droste2002,Borenstein2006},
but except for \cite{Streeter2003}, implications for search performance
under the universal distribution have not been investigated.
The relation between the universal distribution and the NFL theorems
for supervised learning
has been studied by Lattimore and Hutter \cite{Lattimore2011}.
In sequence prediction and reinforcement learning, the universal distribution
has been extensively researched \cite{Hutter2005}.

We first improve on \cite{Streeter2003} by presenting the first easily interpretable theorem that
there is a free lunch if performance is measured in expectation with respect to Solomonoff's universal distribution
rather than the uniform distribution originally used by Wolpert and Macready
(Section~\ref{sec:ufl}).
Unfortunately the size of the free lunch turns out to be somewhat limited. Under only weak assumptions we show
that no computable algorithm can perform much better than random, even when performance is averaged with respect
to the universal distribution (Section~\ref{sec:upbound}). This result is then extended to arbitrary (possibly non-computable) optimisation algorithms for
a commonly used performance measure.

\section{Preliminaries}\label{sec:prelim}

A \emph{(finite binary) string} is a finite sequence
$x = b_1b_2\cdots b_n$ with $b_i \in \SetB = \left\{0,1\right\}$
and length $\ell(x) = n$.
The set of all finite binary strings is denoted by $\SetB^*$.
Strings may be concatenated in the obvious way. Power notation
is used to represent multiple concatenations: for example, $01^40 = 011110$.

A \emph{problem context} is a pair $X,Y$ of finite subsets of $\SetB^*$,
both containing at least $0$ and $1$ (to avoid degenerate cases).
In a problem context $X,Y$, the \emph{search space} is $\sspace$
and the \emph{range} is \range.
We let $\sspaceset$ and $\rangeset$ be the sets of all search spaces
and all ranges respectively.

\begin{definition}
An \emph{optimisation problem} is a collection $\P=\set{P_{XY}}$,
where $\P_\cont$ is a
measure over the finite set $\range^\sspace = \set{f\colon\,\sspace\to\range}$
of functions from \sspace to \range.
\end{definition}

A \emph{search trace $T_n$ on $\probcon$} is an ordered $n$-tuple
$\langle(x_{1},y_1),\dots,(x_{n},y_n)\rangle\in(\sspace\times\range)^n$,
representing a search history. 
The empty search trace will be denoted $\emptrace$.
Let $\tset_n(\sspace,\range)$ be the set of all search traces of length $n$,
and let
$\tset(\sspace,\range)\defined\union_{i=0}^\abs{\sspace}\tset_i(\sspace,\range)$
be the set of traces of any length.
Further, let
$\tset\defined\union_{\probcon\in\sspaceset\times\rangeset}\tset(\sspace,\range)$
be the set of search traces on any context.
If $T_n=\langle(x_{1},y_1),\dots,(x_{n},y_n)\rangle$, then
$T_n^x\defined\langle x_{1},\dots,x_{n}\rangle$  and
$T_n^y\defined\langle y_1,\dots,y_n\rangle$.

An  \emph{optimiser} is a function
$a\colon\,\sspaceset\times\rangeset\times\tset\to\sspaceset$
where $a(X,Y,T)\in X-T^x$ for all $(X,Y,T)$.
The optimiser selects new, unvisited search points in the search
space based on previously seen data and the problem context.
That the optimiser is only permitted to sample unvisited points is standard
in the literature, and non-restrictive in the noise-free setting
considered in this paper.

The setup is this: A problem context $\probcon$ is fixed, and a target function
$f\colon\,\sspace\to\range$ is sampled from $\range^\sspace$ according to the problem
distribution $\P_\cont$. The optimiser is initialised with the empty search
trace and the problem context, and outputs a search point $x_{1}\in\sspace$ by
$a(\sspace,\range,\langle\rangle)=x_{1}$.
The search trace becomes $\langle x_{1},f(x_{1})\rangle$.
The new search trace
is fed to the optimiser, which produces a new search point $x_{2}$ via
$a(\sspace,\range,\langle(x_{1},f(x_{1}))\rangle)$, and so on.
Observe that the search trace is a function of the optimiser, the problem
context and the sampled function $f$. We write $T_{XY}(a,f)$ for
the \quotes{full} trace of length $|X|$ that $a$ generates on $f$ and
$\probcon$; when $X,Y$ is clear from the context, $T(a,f)$ will suffice.
The $\range$-components $T^y_{XY}(a,f)$
will be called the \emph{result vector} of $a$ on $f$ and $\probcon$.
We will also use $R$ to denote result vectors.
Let $\rset(\probcon)$ be the set of all result vectors on \probcon, and
let $\rset=\union_{\probcon\in\sspaceset\times\rangeset}\rset(\probcon)$.

The performance of an optimiser $a$ on a problem $\P$ is measured in
terms of the result vectors it produces.
A function \mbox{$M\colon\,\sspaceset\times\rangeset\times\rset\to[0,\infty)$}
defines a \emph{performance measure} by the $\P$-expected value of $M$ for
$a$ on each problem context \probcon:
\begin{equation}\label{eq:perfdef}
\perfcon(a)\defined\sum_{f\in \range^\sspace}\P_\cont(f)M_\cont(T^y(a,f))
\end{equation}

We use $\ind{s}$ for the Iverson bracket that is 1 when $s$ is true, and 0
otherwise.
For any list $R$, $R[i]$ extracts its $i$th element.

\section{The universal distribution}\label{sec:unidist}

We now give a short introduction to Kolmogorov complexity and the universal
distribution. Detailed references are \cite{Li2008} and \cite{Cal02}.

Prefix codes are central elements in algorithmic information theory.
A \emph{prefix code} is a set of  \emph{code words} (formally, strings)
where no code word is a prefix of another.
This makes prefix codes \emph{uniquely decodable}:
in a sequence of appended code words it is possible to tell where
one code word ends another begins.
Kraft's inequality gives a lower bound on the length of the code words
in a prefix code \cite[p.~76]{Li2008}.
Definition~\ref{def:strenc} gives some
commonly used prefix encodings of strings, numbers, lists and functions.

\begin{definition}[String encodings]\label{def:strenc}
Let $x$ be a binary string, $n$ a natural number and
$Z=z_1,\dots,z_n$ a list of strings.
Then $\bar x \defined 1^{\ell(x)}0x$, $\bar n\defined 1^n0$ and
$\bar Z\defined \bar n \bar z_1 \cdots \bar z_n$  defines
prefix codes for $x$, $n$ and $Z$. The code for lists
may be applied recursively to lists of lists. Target functions $f\colon\,X\to Y$
are encoded by lists $f(x_1),\dots,f(x_n)$ where $x_1,\dots,x_n$ are
the elements of $X$ in order.
\end{definition}

For technical reasons, regular Turing machines are not suitable for
defining the universal distribution, so prefix machines are
often used instead.
\begin{definition}[Prefix Machines]
A \emph{prefix machine} $V$ is a Turing machine with one
 unidirectional input tape,
one unidirectional output tape, and some bidirectional work tapes.
Input tapes are read only, output tapes are write only.
All tapes are binary and work tapes are initialised with zeros.
We say $V$ halts with output $x$ on input $p$ given $s$ and write $V(p|s)=x$,
if $\bar s p$ is to the left of the input head
and $x$ is to the left of the output head when $V$ halts.
For any $s\in \SetB^*$, the inputs on which $V(\cdot|s)$ halts
form a prefix code.
Also, just as for regular Turing machines, there are
\emph{universal prefix machines} that can simulate
any other prefix machine.
\end{definition}

\begin{definition}[Prefix Complexity]\label{def:K}
Let $x,y \in \SetB^*$ be finite binary strings and $U$ a universal
prefix machine, then the \emph{Kolmogorov complexity of $x$ conditioned on
$y$} is the length of the shortest program that given $y$ outputs $x$.
\begin{equation}\label{eq:defK}
K_U(x|y) \defined \min\set{\ell(p): U(p|y) = x}
\end{equation}
A simple but fundamental theorem is that $K_U$ depends on $U$ only up
to constant factors, so
from now on, as is usual in algorithmic information theory, we fix
an arbitrary universal prefix machine as a \emph{reference machine}
and simply write $K(x)$ for $K_U(x)$.
\end{definition}

\begin{definition}[Function complexity]
Let $f\colon\, X\to Y$, then the \emph{complexity of $f$} is $K(f|X,Y)$,
with $f$ and $X,Y$ encoded by strings according to Definition~\ref{def:strenc}.
\end{definition}

The Martin-Löf--Chaitin Thesis states that randomness may be defined as
incompressibility \cite[p.~705]{Gabbay2011}. A target function is
 \emph{incompressible} or  \emph{random} if
$K(f|\sspace,\range) \geq |\sspace|\log |\range|$.
A classical result in algorithmic information theory shows that almost all
functions are (nearly) random.
Thus, the uniform distribution puts the  majority of its weight on random
functions, which is one explanation for why it is hard to optimise under
the uniform distribution.
In contrast, the universal distribution puts more weight on ``simple'',
non-random functions:
\begin{definition}[Universal distribution]\label{def:unidist}
For each context $X,Y$, the \emph{universal distribution} is defined as
\begin{align}
\label{eq:m1}
\mf(f) := c_\mf\cdot 2^{-K(f|\sspace,\range)},
\end{align}
where $c_\mf=1/\sum_{f\colon\sspace \to\range} 2^{-K(f|\sspace,\range)}$ is just a
normalising constant.
In the literature, unnormalised versions of \m are often considered.
Although $c_\mf$ may fluctuate with $X,Y$, there is a
constant $c_\m$ depending only on the reference machine such that
$1\leq c_\mf\leq c_\m$ for all $X,Y$.
Note that \m is an optimisation problem, since $\mf$ is a distribution
over $Y^X$ for each $X,Y$.
\end{definition}

Somewhat surprisingly, there is an equivalent definition of \m as the
distribution obtained by feeding random coin-flips into a universal
prefix machine with access to $X,Y$.
\begin{align}
\label{eq:m2}
\m_{XY}(f)\approx \sum_{p\in \SetB^*}2^{-\ell(p)} \ind{V(p|X,Y)=f}.
\end{align}
The approximation holds up to irrelevant multiplicative factors, so (\ref{eq:m2}) is often
used in place of (\ref{eq:m1}) as the definition of the universal distribution.\footnote{Even the definition
given in (\ref{eq:m1}) depends on the choice of reference machine up to multiplicative factors.}
Feeding a universal prefix-machine random coin-flips is a natural
formalisation of the uniform prior over computer programs advocated
in the introduction. Thus \m may be justified as a
\emph{subjective prior} for the assumption that
the target function has been computably generated.\footnote{
There are also ``objective'' grounds to prefer \m as a prior, including
\emph{regrouping invariance} \cite{Hutter2007} and \emph{dominance}
\cite{Li2008}. Neither hold for the uniform distribution.}

The bias away from randomness also aligns with our intuition of how
functions ``ought to be optimised''. If the first hundred observations are
predicted by a simple polynomial, then common sense (and Occam's razor)
suggests that the best prediction of unseen points is that they follow the
polynomial. In general, the ``simplest'' structure perceivable in the
data should be the most likely extrapolation. The universal distribution
is consistent with this intuition.
A detailed discussion of the philosophical justification for the universal
prior can be found in \cite{Rathmanner2011}.

To summarise, we have argued for the universal distribution on the
following grounds:
\begin{itemize}
\item (Weak assumptions): If the target function is generated by a computer
program, then a uniform prior over computer programs is justified.
Formalised as in \eqref{eq:m2}, this yields the universal distribution.
\item (Downweighs randomness): A uniform prior over target functions
puts the (vast) majority of the weight on (nearly) random functions.
The universal distribution concentrates on structured functions, without
favouring any particular class of functions.
\item (Aligns with Occam): Intuition and Occam's razor
suggests that the best extrapolation is the continuation of
the ``simplest'' pattern observable in data, which corresponds well
with the relative probabilities of the universal distribution.
\end{itemize}
Next we will present some background on the NFL theorems
and introduce our performance measure \mptm, before taking a closer
look at optimisation under \m.

\section{No Free Lunch}
The NFL theorems provide important insights into the possibility of
universal optimisation.
They show that for certain distributions $\P_\cont$ all
optimisers perform identically with respect to
some (or all) performance measure(s). This is often phrased as
\quotes{there is no free lunch available for $\P_\cont$}.
For example, if NFL holds for a performance measure depending on how many
function evaluations are required to find the maximum, then
this implies that in expectation a hill-climbing optimiser
will find the maximum as slowly as a hill-descending optimiser.

\begin{definition}[Performance measure-NFL]\label{def:pmnfl}
\emph{NFL} holds for a distribution $\P_\cont$ and a performance measure $M$
 if $\perfcon(a)=\perfcon(b)$ for all optimisers
$a$ and $b$. If NFL holds for \emph{all} performance measures $M$, then
NFL simply holds for $\P_\cont$.
If NFL does not hold for $\P_\cont$ (and $M$), then we say that there is
\emph{free lunch} for $\P_\cont$ (and $M$).
\end{definition}

The stronger statement that NFL holds for all performance measures may
equivalently be defined in terms of result vectors.
The proof of the equivalence is a straightforward application
of the definitions.

\begin{lemma}[Result vector-NFL]\label{le:rvnfl}
Let $\P_{\cont a}(R)$ 
be the probability $\sum_{f\in\range^\sspace}\P_\cont(f)\ind{T^y(a,f)=R}$
that an optimiser $a$ generates the result vector $R$,
then NFL holds for $\P_\cont$ if and only if $\P_{\cont a}(R)=\P_{\cont b}(R)$
for every pair of optimisers $a$ and $b$ and every result vector
$R\in Y^{|X|}$.
\end{lemma}

\subsection{The NFL theorems}
Igel and Toussaint \cite{Igel2004} showed that the precise condition for NFL
is \emph{block uniformity} of $\P_{\cont}$.
\begin{definition}[Block uniformity]\label{def:block}
A \emph{histogram} for a function $f$ is a function $h_f\colon\,\range\to\SetN$
defined as $h_f(y)=|f^{-1}(y)|$,
indicating how many $x$'s map to every $y$.
The subset of all functions $X\to Y$ with histogram $h$
is called the \emph{base class} of $h$, and is denoted $B_h$.
The distribution $\P_{XY}$ is \emph{block uniform} if for every
$h$ it holds that $f,g\in B_h\implies \P_\cont(f)=\P_\cont(g)$.
\end{definition}

\begin{theorem}[Non-uniform NFL \cite{Igel2004}]\label{th:nonuninfl}
NFL holds for $\P_\cont$ if and only
if $\P_\cont$ is block uniform.
\end{theorem}

The original NFL theorem by Wolpert and Macready \cite{Wolpert1997} showed
that NFL holds when $\P_{\cont}$ is uniform on $\range^\sspace$. As uniform
distributions are a special case of block uniform distributions, Wolpert
and Macready's result follows from Igel and Toussaint's.

Another special case is the NFL theorem for uniform optimisation problems
over function classes closed under
permutation (c.u.p.)\ by Schumacher et al.\ \cite{Schumacher2001}.
A \emph{permutation} is a bijective function $\sigma\colon\,X\to X$ that permutes
functions via $(\sigma\! f)(x)=f(\sigma^{-1}(x))$.
A class $F\subseteq \range^\sspace$ is \emph{c.u.p.}\ if
$f\in F\implies \sigma\!f\in F$ for all permutations $\sigma$.
The uniform distribution $u_F$ over $F$ is defined
by $u_F(f) \defined 1/|F|$ if $f \in F$ and $0$ otherwise.

\begin{theorem}[NFL for c.u.p.\ classes \cite{Schumacher2001}]\label{th:cupnfl}
If $\P_\cont$ is the uniform distribution over a class
$F\subseteq \range^\sspace$,
then NFL holds for $\P_\cont$ if and only if $F$ is c.u.p.
\end{theorem}

A simple consequence of the NFL theorems is that all optimisers produce
the same result vectors. We state this as a lemma for future reference.
\begin{lemma}[\cite{Schumacher2001}]\label{le:numfun}
The set of result vectors
$\set{R\in\rset(X,Y): T^y(a,f)=R\text{ for some }f\in Y^X}$
ever produced by an optimiser $a$ on $X,Y$ is the same for all optimisers.
\end{lemma}

The NFL theorems have also been investigated in infinite and continuous
domains. Depending on the generalisation, free lunches may or may not emerge
in those settings \cite{Auger2007, Rowe2009}.

\section{Performance measures}\label{sec:partmeas}
So far we have only considered problems for which either NFL holds for
all performance measures, or for which a free lunch is
available for some performance measures.
Often, however, we are interested in performing well under a fixed
performance measure of interest.
One natural choice of performance measure  is \emph{optimisation time},
which in this context means the number of function evaluations required
to find the maximum.\footnote{
In black-box optimisation in general, and evolutionary algorithms in particular,
the evaluation of the target function typically constitutes the main expense of
computation time. This is the motivation behind the name
\emph{optimisation time} for the number of target function
evaluations.}

\begin{definition}
The \emph{optimisation time performance measure} \mptm is defined as
\begin{align*}
\mptmcon(R)\defined&\;\min_i\left(R[i]=\max Y\right),\\
\mptmcon^\P(a)\defined&\sum_{f\in\range^\sspace}\P_\cont(f)\mptmcon(T^y(a,f))
\end{align*}
for result vectors and optimisers, respectively.
Under $\mptm$ a low score is  better than a high score.
\end{definition}

A variety of performance measures have been considered in the literature.
Some use properties of the $k$ first function evaluations, for example
the number of values exceeding a certain threshold
\cite{Christensen2001,Whitley2006,Jiang2011}, or
the probability that some seen value exceed the threshold
\cite{Wolpert1997}.
Griffiths and Orponnen
\cite{Griffiths2005} use a performance measure $M_{\max}$ depending on
the size of the greatest value of the first $k$ observations.
Others, such as \cite{Borenstein2006, Jansen2013}, use \mptm.
The main reason we prefer \mptm to the other alternatives
is that it is better suited for the asymptotic results we will aim for.

Results about particular performance measures often have
greater practical interest than their arbitrary-measure counterparts.
In addition, particular performance measures may also have theoretical
interest.
Under particular performance measures, NFL may
hold for  classes that are not c.u.p.\ \cite{Griffiths2005}.
This does not contradict Theorem \ref{th:cupnfl},
which only claims that for every non-c.u.p.\ class, there is
\emph{some} performance measure permitting a free lunch.
Indeed, it is unsurprising that NFL will apply to wider ranges of function
classes when a fixed performance measure is used.
The conditions for NFL under Griffiths and Orponnen's performance
measure $M_{\max}$ turn out to be significantly more intricate compared
to the standard NFL case \cite{Griffiths2005}.

Another difference is found in the ``cleverness'' required to exploit a
free lunch.
Optimisers that choose the next point to probe irrespective of
previous observations are called \emph{non-adaptive}; such optimisers
can only exhibit a limited amount of sophistication.
Proposition \ref{pr:enum} shows that when a free lunch is available
and arbitrary measures are allowed, then
there is free lunch for a non-adaptive optimiser.
In contrast, under particular
measures such as $\mptm$ and $M_{\max}$, adaptive optimisers may
differ in performance while all non-adaptive optimisers perform the same.
In this sense, ``smarter'' algorithms may be required for exploiting a free
lunch when using  a particular performance measure,
compared to when arbitrary performance measures are permitted.

\begin{proposition}\label{pr:enum}
If NFL does not hold for a distribution $\P_\cont$,
then there is free lunch for a non-adaptive optimiser under
some performance measure.
\end{proposition}
\begin{proof} Since NFL does not hold for $\P_\cont$ we have by
Theorem \ref{th:nonuninfl} that $\P_\cont$ is not block uniform.
Hence there are two functions $f$ and $\sigma\! f$ in the same base class
$B_h$ such that $\P_\cont(f)>\P_\cont(\sigma\! f)$,
where $\sigma$ is a permutation on $\sspace$.
Let $e$ and $e_\sigma$ be non-adaptive optimisers, with $e$ searching
$\sspace=\set{x_1,\dots,x_n}$ in order,
and $e_\sigma$ searching $X$ in the order
of $\sigma\! X=\set{\sigma(x_1),\dots,\sigma(x_n)}$.
Now $e$ generates the result vector
$R_f=\langle f(x_1),\dots,f(x_n)\rangle$ exactly when $f$ is the
true function, and $e_\sigma$ generates $R_f$ exactly when $\sigma\!f$
is the true function.
An immediate consequence is that
$\P_{\cont e}(R_f)=\P_\cont(f)>\P_\cont(\sigma\!f)=\P_{\cont e_\sigma}(R_f)$.
 That is, the non-adaptive algorithms $e$ and $e_\sigma$ generate $R_f$ with
different probability, which means that there is free lunch for some
non-adaptive optimiser under some performance measure by Lemma
\ref{le:rvnfl}.
\end{proof}

In conclusion, specific performance measures can be considered for both
practical and theoretical reasons. They are more practically relevant in
the sense that they measure aspects we care about in practice (such
as how long it takes to find a maximum). But they also have theoretical
interest, as they expose aspects that are invisible from an
arbitrary-measure perspective.

\section{Universal Free Lunch}\label{sec:ufl}

We now turn to the question of whether or not a free lunch is available
under \m, which we will answer in the affirmative for both
arbitrary performance measures and \mptm.

The universal distribution solves the induction problem for sequence
prediction \cite{Hutter2005,Rathmanner2011}.
Black-box optimisation also include an induction problem in the extrapolation
of target-function behaviour from the points already evaluated.
Although successful inference of the  target-function
behaviour may not be strictly necessary,
it will typically enable better choices of future search points.

There are several important differences between sequence prediction and
optimisation.
First, optimisation is an \emph{active} setting: the choices
of the optimiser affect both the learning outcome and the reward.
This entails an exploration/exploitation tradeoff in the choice between
potentially informative points and
points likely to mean high performance (e.g.\ points likely to be a maximum).
Further, optimisation is a finite setting,
which yields less time to exploit a good model (compared to
sequence prediction where infinite sequences are considered).
There are also major differences in the hypothesis classes and in how
performance is measured.

Section \ref{sec:quantfree} presents a number of results bounding
the amount of free lunch under \m. Perhaps surprisingly, only a small
amount of free lunch is available under the universal distribution.

\subsection{Free lunch under arbitrary performance measures}\label{sec:unifreearb}
Streeter \cite{Streeter2003} showed that there is free lunch for $\m$ under
certain technical conditions.
We prove a similar result, but with more interpretable conditions
(in terms of the size of \sspace, only).
We also use a different proof than Streeter.

\begin{theorem}[Universal free lunch]\label{th:freem}
There is free lunch for the problem \m for all problem contexts with
sufficiently large search space
(the required size depending on the reference machine only).
\end{theorem}

\begin{proof}
It will be shown that $\m_\cont$ is not block uniform for problem contexts
with sufficiently large \sspace, which by Theorem \ref{th:nonuninfl} implies
that NFL does not hold.
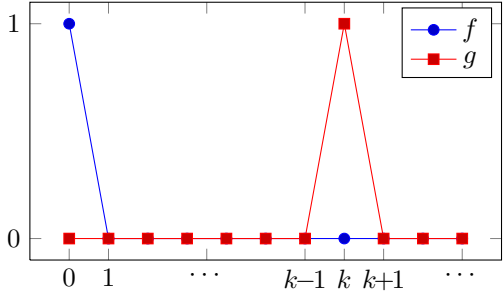
\begin{figure}
\centering
\begin{tikzpicture}
\begin{axis}[
  height=5cm,
  ytick=data,
  xtick={0,1,3.5,6,7,8,10},
  xticklabels={0,1,$\cdots$,$k\sminus 1$,$k$,$k\splus1$,$\cdots$}]
\addplot coordinates {
	(0,1)
	(1,0)
	(2,0)
	(3,0)
	(4,0)
	(5,0)
	(6,0)
	(7,0)
	(8,0)
	(9,0)
	(10,0)  };
\addplot coordinates {
	(0,0)
	(1,0)
	(2,0)
	(3,0)
	(4,0)
	(5,0)
	(6,0)
	(7,1)
	(8,0)
	(9,0)
	(10,0)  };
	\legend{$f$, $g$}
\end{axis}
\end{tikzpicture}
\caption[Illustration of non-block uniformity of \m]{The function $f$
has complexity bounded by a constant $c_f$ independent of \sspace and
\range. In contrast, the complexity of $g$ grows logarithmically with
$\abs\sspace$. See the proof of Theorem \ref{th:freem} for details.}
\label{fig:plotunilunch}
\end{figure}

Pick a problem context \probcon.
Consider two functions $f$ and $g$ in the base class $B_h\subseteq Y^X$
of functions
with one value 1 and the rest of the values 0. Let $f$ be 1 at $x_1$
and let $g$ be 1 at some point $x_k$ chosen so that
$\cK{g}{\probcon}\geq \lb |X|-1$  (see Fig.~\ref{fig:plotunilunch}).
To see that such a $g$  exists,  note that there are $|X|$
different functions in $B_h$.
 As the halting programs for the reference machine
constitute a prefix code, there can be at most  $|X|/2$ halting programs
of length $\leq \lb |X|-1$ by Kraft's inequality.
Thus at least one of the $B_h$-functions must have a shortest program
longer than $\lb |X|-1$, and therefore
complexity $\fK{g}\geq \lb |X|-1$.
Meanwhile, $\fK{f}\leq c_f$ for some constant $c_f$ independent
of the problem context.
So for search spaces with $\lb(\abs\sspace)-1>c_f$,
this means that $f$ will have lower complexity than $g$, and
thus that $\mf$ will assign different probabilities to $f$ and $g$.
As $f$ and $g$ are elements of the same base class, this shows
that $\m$ is not block uniform for search spaces greater than
$2^{c_f+1}$.
By Theorem \ref{th:nonuninfl}, this implies a free lunch for
\m under some performance measure.
\end{proof}

Indeed, $\m$ is not even close to block uniform for large search spaces in the
sense that the functions of type $f$ and $g$ will receive substantially
different weights.
However, this does not necessarily imply a big free lunch, as
we shall see in Section \ref{sec:quantfree}.

\subsection{Free lunch under $\mptm$}\label{sec-freemptm}
As has been discussed, in practice we often care about a particular performance
measure such as \mptm.

\begin{theorem}\label{th:freemptm}
There is free lunch for the problem $\m$ under the performance measure
$\mptm$ for all problem contexts with sufficiently large search space
(the required size depending on the reference machine only).
\end{theorem}

The proof is similar to Theorem \ref{th:freem}, but more work is required
to ensure a complexity difference between two potential maximums, rather
than between two specific functions. A full proof is included in the
Appendix.

\section{Upper Bounds}\label{sec:upbound}
\label{sec:quantfree}

Theorems \ref{th:freem} and \ref{th:freemptm} show that there is free lunch
under the universal distribution. This section will bound the amount of
free lunch available, and show that it is only possible to outperform
random search by a constant factor.
First we show that the performance of computable optimisers deteriorates
linearly with the worst-case scenario and
the size of the search space.
This result applies to decidable performance measures in general, and has a
concrete interpretation for $\mptm$, where it implies that as the size of the
domain is increased, a non-zero fraction of the domain must be probed before a
 maximum is found in expectation.
This does not contradict the free lunches above, as the required fraction
may differ between optimisers.

We also consider possible ways to circumvent the negative result described
above by means of incomputable search procedures.
A further negative result for $\mptm$ is obtained: It does not appear possible
to find the maximum with only $o(|\sspace|)$ target function evaluations.
That is, the expected number of probes required to find the maximum grows
linearly with the size of the search space, but again, the proportion may
differ substantially between optimisers.

\subsection{Computable optimisers}\label{sec:compalg}
To bound the amount of free lunch available for computable optimisers,
we will adapt a proof-technique for showing that average-case complexity
is equal to the worst-case complexity under the universal distribution
\cite[Section~4.4]{Li2008}.
Although no formal theorem relies on it, we will think of greater
$M$-values as \emph{worse} performance.

\begin{lemma}
A function \mbox{$f_{\rm bad}\colon\,\sspace\to\range$} is
\emph{maximally bad} for an optimiser $a$ on the problem context
\probcon with respect to a performance measure $M$ if
$M_\cont(T^y(a,f_{\rm bad}))=\max_{R\in\rset(\probcon)}M_\cont(R)$.
There always exists a maximally bad function
$f_{\rm bad}\colon\,\sspace\to\range$ for $a$ with respect to $M$,
regardless of the performance measure $M$, the optimiser $a$ and the
problem context \probcon.
\end{lemma}
\begin{proof}
By Lemma \ref{le:numfun}, all optimisers produce the same result
vectors, so it suffices to show that \emph{some}
optimiser has a maximally bad function.
The non-adaptive optimiser $e$ that searches $X$ in order has a maximally bad
function.
To see this, let $R_{\rm bad}$ be a maximally bad result vector on \probcon,
and let $f$ be the function satisfying $f(x_i)=R_{\rm bad}[i]$ for all
$x_i\in X$.
Then $e$ produces $R_{\rm bad}$ on $f$.
\end{proof}

A performance measure $M$ for which there is an algorithm
deciding whether $M_\cont(R_1)<M_\cont(R_2)$ for every \probcon
and every pair of result vectors $R_1,R_2\in\rset(\probcon)$ is
\emph{decidable}.
For any decidable performance measure $M$, it is possible to
create a procedure $\procedure{FindWorst}(a,\sspace, \range)$ that given a computable
optimiser $a$ (specified by some binary string) and a context $X,Y$,
returns a maximally bad function $\fbad\colon\,X\to Y$ for $a$.
$\procedure{FindWorst}$ is a computable operation since $a$ is computable and
$M$ is decidable: \procedure{FindWorst} need only simulate $a$ on all possible
functions in $\range^\sspace$, and output one that yields a worst result vector.
This shows that for all decidable performance measures, all computable
optimisers $a$ and all contexts $X,Y$, there is a maximally bad
function \mbox{$\fbad\colon\,X\to Y$}
for $a$ with complexity
\begin{equation}
\cK{\fbad}{\sspace,\range}\leq \length(\procedure{FindWorst})+\length(a)+c
\enspace,\label{eq:fbad}
\end{equation}
where the $c$ term depends only on the reference machine, and absorbs
the cost for initialising \procedure{FindWorst} with $a$,
$\sspace$ and $\range$.
Pivotally, the bound is independent of $X,Y$. This is the central observation
behind the following theorem, which shows that expected performance always
deteriorates linearly with the worst-case scenario.
The theorem's prime relevance is for performance measures whose
worst-case value grows with $X$.

\begin{theorem}[Almost NFL for $\m$]\label{th:badalg}
For every decidable performance measure $M$ and every computable optimiser $a$
there exists a constant $c_a>0$ such that for all $X,Y$
\[M_{\cont}^\m(a)\geq c_a\max_{R\in\rset(\probcon)}M_{\sspace\range}(R)\enspace.\]
\end{theorem}
\begin{proof}
Let \probcon be a problem context and
$\fbad$ be the output of $\procedure{FindWorst}(a,\sspace,\range)$, then
\begin{align*}
M_\cont^\m(a)=&\sum_{f\in Y^X}\mf(f)M_\cont(T^y(\range,a,f))\\
&\geq c_\mf 2^{-K(\fbad|\sspace,\range)}M_{\sspace\range}(T^y(a,\fbad))
\enspace,
\end{align*}
where we have first used the definition \eqref{eq:perfdef} of performance
measures,
and then that the sum of non-negatives is greater than all of its terms.
But $c_\mf 2^{-K(\fbad|\sspace,\range)}\geq c_a$ for some $c_a>0$ independent of
$X,Y$ due to $c_{\m_{XY}}\geq 1$ and the complexity bound \eqref{eq:fbad}.
And $T^y(a,\fbad)$ was a worst result vector by the construction of
\procedure{FindWorst}.
Combined, this gives the bound
$M_\cont^\m(a)\geq c_a\cdot\max_{R\in\rset(\probcon)}M_{\sspace\range}(R)$.
\end{proof}

This theorem shows that for every performance measure $M$, there is only a
 constant amount of free lunch available in an asymptotic sense.
It has no impact on performance measures whose worst-case value does not
grow unboundedly with either $\sspace$ or $\range$.
However, the \quotes{semi-assumption} of higher values being
worse is not necessary: If the converse is the case and high values are better,
then the proposition shows that all optimisers will do well. Indeed, this is
also an NFL result, as it implies that random search (and even optimisers
designed to do poorly!)\ will perform well.

Applied to the performance measure $\mptm$, Theorem \ref{th:badalg} has a
fairly concrete interpretation: For any computable optimiser
$a$, the expected number of evaluations to find the maximum grows linearly with
$|\sspace|$.
Corollary \ref{co:badalgmptm} follows immediately from Theorem
\ref{th:badalg} and the observation that
for any context $X,Y$, the worst-case scenario is to find a maximum only
at the very last probe; that is, $\max_{R\in\rset(X,Y)}\mptmcon(a)=|X|$.

\begin{corollary}\label{co:badalgmptm}
For every computable optimiser $a$ there exists a constant $c_a>0$ such
that $\mptmcon^\m(a)\geq c_a\cdot\abs\sspace$ for all optimisers $a$ and
all problem contexts \probcon.
\end{corollary}

The implications of this result should not be overstated. The constant $c_a$
 may be very small; for example, if the description of the optimiser $a$
is 100 bits long, then $c_a$ becomes of the order $2^{-100}$. The fact that
the expected number of probes is forced to grow linearly with such a constant
is mainly of theoretical importance.
Nonetheless, the result does illustrate a fundamental hardness of optimisation,
 and shows that the universal distribution does not provide enough bias for
efficient (sublinear) maximum finding.

\subsection{Needle-in-a-haystack functions}\label{sec:niah}
A problematic class of functions is the class of so-called
needle-in-a-haystack (NIAH) functions.
We will use them to generalise Corollary \ref{co:badalgmptm} to
incomputable optimisers.
A \emph{NIAH-function} is a target function
that is $0$ in all points except one where it equals~1.
The exception point is called the \emph{needle}.
It should be intuitively clear that it is hard to find the maximum of a
NIAH-function. Probing a NIAH-function, the output will generally just
turn out to be 0 and provide no clues to where the needle might be.

Formally, for any $X,Y$
let $\niah_{\sspace\range}$ be the class of NIAH functions on $X,Y$ and
let $\uniah$ be the \emph{uniform  NIAH problem}
defined as $\niahp(f)\defined 1/|\niahc|$ if $f\in \niahc$ and 0 otherwise.
The function class $\niah_{\sspace\range}$ is c.u.p. for any $X,Y$, so NFL
holds for $\niahp$ by Theorem \ref{th:cupnfl}.
The expected performance (of any optimiser) on the uniform
NIAH-problem can be calculated
from a general result of Igel and Toussaint \cite{Igel2003a}. They show
that for any c.u.p.\ problem $u_F$ where $F$ only contains functions with
exactly $m$ maxima, the expected number of probes to find a maximum is
\mbox{$(\abs{\sspace}+1)/(m+1)$}. The NIAH-functions have exactly one maximum,
which gives the following lemma.
\begin{lemma}\label{le:niahexp}
Under \niahp, the expected optimisation time
is $\mptmcon^{\uniah}(a)=(\abs{\sspace}+1)/2$
for any optimiser $a$.
\end{lemma}

One feature that makes the NIAH-class more problematic than other c.u.p.\
function classes is that the NIAH-functions all have fairly low complexity
(as remarked by \cite{Schumacher2001,Borenstein2006}). The NIAH-functions
have low complexity, since to encode a NIAH-function one only needs to
encode that it is NIAH (which takes a constant number of bits) and the position
of the needle (which requires at most $\lb\abs{\sspace}$ bits).
A NIAH-function thus has complexity of order $\Ordo(\lb\abs{\sspace})$;
in comparison,
a random function has exponentially greater complexity (above
$\abs{\sspace}\lb\abs{\range}$).

The NIAH-measure is computable. This is intuitively obvious, and easily
verified against the formal definitions of computable functions.
Definitions of real-valued computable functions can be found in \cite{Li2008}.
It is well-known that \m \emph{dominates} any computable measure
in the following sense.
\begin{lemma}[\m dominates $\uniah$]\label{le:mdomu}
There is a constant $c_{\rm NIAH}>0$ such that for all $X,Y$ and all
functions $f\colon\,\sspace\to\range$, it holds that
$\mf(f)\geq c_{\rm NIAH}\cdot \niahp(f)$.
\end{lemma}

\subsection{Incomputable optimisers}\label{sec:upboundinc}
Theorem \ref{th:badalg} and Corollary \ref{co:badalgmptm} were proven
for computable optimisers.
We now show that even incomputable optimisers suffer a linearly growing
loss in $|X|$ when the performance measure is $\mptm$.
 Incomputable search
procedures may seem like remote objects of concern, but for example the
(Bayes-)optimal procedure for $\m$ is incomputable due to the
incomputability of $\m$. Therefore, incomputable procedures do at least
have theoretical interest.

The following theorem generalises Corollary \ref{co:badalgmptm} to incomputable
search procedures, showing that they also must search a linearly growing
portion of $\sspace$ to find the maximum.
The theorem does not generalise to arbitrary performance measures
however, so the analogous generalisation of Theorem \ref{th:badalg} may
not be true.

\begin{theorem}[Almost NFL for $\m$ and $\mptm$]\label{th:unianfl}
Under \m, the expected optimisation time grows linearly with $\abs\sspace$,
regardless of the optimisation strategy.
\end{theorem}

\begin{proof}
The dominance of \m over $\uniah$ is used in \eqref{eq:dom},
between an expansion \eqref{eq:exp} and a contraction \eqref{eq:cont}
according to the definition \eqref{eq:perfdef} of performance measures:
\begin{align}
\mptmcon^\m(a)
=&\sum_{f\in\range^\sspace}\!\mf(f)\mptmcon(T^y(a,f))\label{eq:exp}\\
\geq c_{{\rm NIAH}} &\sum_{f\in Y^X}\!\niahp(f)\mptmcon(T^y(a,f))\label{eq:dom}\\
= c_{{\rm NIAH}}& \cdot\mptmcon^{\uniah}(a)\label{eq:cont}
\end{align}
Lemma \ref{le:niahexp} established \eqref{eq:cont} to be
$c_{{\rm NIAH}}\cdot (\abs{\sspace}+1)/2$ for all optimisers $a$.
Thus the expected \mptm-performance is always bounded below by
$c_{{\rm NIAH}}\cdot (\abs{\sspace}+1)/2$, which grows linearly with $|X|$.
\end{proof}

Since optimisation time can never grow faster than linearly with $|X|$,
the bound is asymptotically tight.
In this sense, Theorem \ref{th:unianfl} may be viewed as an asymptotic
almost-NFL theorem for the universal distribution and $\mptm$.
The constant $c_\niah$ in the proof may be very small however,
so Theorem \ref{th:unianfl} does not rule out that expected optimisation time
differ substantially between optimisers.

\section{Conclusion}

In this paper we investigated the No Free Lunch theorems when the performance
of an algorithm is measured in expectation with respect to Solomonoff's
universal distribution. We showed in Theorem \ref{th:freem} that there
is a free lunch with respect to this distribution.

Somewhat surprisingly, despite the bias away from randomness exhibited by
the universal distribution, the
size of the free lunch turns out to be quite small, at least asymptotically
(Theorems \ref{th:badalg} and \ref{th:unianfl}).
The reason for this is that there are many functions that are both simple
and hard to optimise. Most notably the needle-in-a-haystack
functions, which have complexity of at most $O(\log |X|)$, but for
which a maximum cannot be found without $O(|X|)$ probes.

It should be emphasised that there is little need to be too gloomy about
the negative results. The upper bounds on the size of the free lunch
given in both negative theorems depend on constants that in practise
are likely to be very small. Optimisation is a hard problem, so we should
not be too surprised if there are some reasonably frequently occurring
functions that cannot be efficiently optimised.

The fact that simplicity is not a sufficient characterisation of the difficulty of optimising a function is unfortunate. This is not true
in other domains such as supervised learning and sequence prediction where approaches based on Solomonoff's universal prior are theoretically
optimal in a certain sense \cite{Hutter2005}. One difficulty of optimisation lies in the exploration/exploitation problem, which occurs because at each
time-step an optimisation algorithm must make a choice between trying to learn the true function and probing the point that it believes to be the maximum.

Since Kolmogorov complexity is (by itself) insufficient for characterising the difficulty of optimising a function, a new criterion
is required. We are currently unsure what this should look like and consider this interesting future research.

\bibliographystyle{alpha}
\bibliography{library}

\appendix
\addcontentsline{toc}{section}{References}
\section*{Appendix}
\addcontentsline{toc}{section}{Appendix}
We here include a proof of Theorem \ref{th:freemptm}. The proof builds on
the following definitions and lemmas.

\begin{definition}
A point $x\in X$ is \emph{incompressible} with respect to the context
$X,Y$ if $K(x|X,Y)\geq \log(|X|)$.
\end{definition}

At least half of the points of any search space will be incompressible.
Functions that only have incompressible maxima (except, possibly, for a maximum
at $x_1$) will play an important role since they are guaranteed to have high
complexity. The reason for excluding $x_1$ will be apparent in the proof of
Theorem \ref{th:freemptm}.

\begin{lemma}\label{le:incfunc}
Let $X,Y$ be a problem context, and let $D\subseteq X$ be a
non-empty set of incompressible points. Let $g\colon\,X\to Y$ have at least one
maximum in $D$, and no maximum outside $D\cup\{x_1\}$.
Then $\cK{g}{\sspace,\range}\geq \lb(|\sspace|)-c$, where $c$ depends only on the reference machine and not on $g$, $\sspace$ or \range.
\end{lemma}

\begin{proof}
Let $g$ be as in the Lemma statement, and
let $x_m\in\sspace-\{x_1\}$ be the first maximum of $g$ not at $x_1$.
Then $x_m$ can be coded by means of $g$ with constant length procedure
$\procedure{FirstMax}(g)$ that computes the first maximum not at $x_1$ for a
given function $g$. Hence $\pK{x_m}\leq \fK{g}+\length({\procedure{FirstMax}})+c$.
The constant $c$ depends only on the reference machine, and absorbs the cost
of initialising \procedure{FirstMax} with a provided description of $g$.

By assumption $x_m$ was incompressible, so $\pK{x_m}\geq \lb{|X|}$.
Combined and rearranged, this gives
$\fK{g}\geq \lb{|\sspace|} -\length({\procedure{FirstMax}})-c$.
The lemma now follows by absorbing $\length({\procedure{FirstMax}})$ into $c$.
\end{proof}

We are now ready for the proof of Theorem \ref{th:freemptm}
that shows that there is
free lunch for \mptm on the problem $\m$.
The key idea is to show that there is a trace after which two unexplored
points have different probability of being the maximum.

\def\thetheorem{\ref{th:freemptm}}
\begin{theorem}
There is free lunch for the problem $\m$ under the performance measure
$\mptm$ for all problem contexts with sufficiently large search space
(the required size depending on the reference machine only).
\end{theorem}
\addtocounter{theorem}{-1}

\begin{proof}
Let $k\geq 2$ and let \probcon be a problem context with $\abs\sspace\geq 2k$.
Let $D_k\subseteq X$ be of size $k$ and only include incompressible points.
Let $Q=\sspace-D_{k}-\{x_1\}$.
Let $G=\{g\in\range^\sspace\colon\,x\in Q\implies g(x)=0\}$ contain all functions
that are 0 on $Q$.
Let $f$ be 0 everywhere, except at $x_1$, where $f$ is 1.
The complexity of $f$ is upper bounded by a constant $c_f$ independent
of $X$. Since $f\in G$, we get $\mf(G)\geq\mf(f)\geq 2^{-c_f}$.

Let $x_m\in D_k$ be an incompressible point, and let
$G_m=\{g\in G\colon\,g(x_m)=\max g\}$. As the functions in $G$ are all 0
on $Q$, the cardinality of $G_m$ is at most $|Y|^{|X-Q|}=|Y|^{k+1}$.
Also, the functions in $G_m$ all have complexity above $\log|X|-c$ for
some $c$ independent of $X,Y$, by Lemma \ref{le:incfunc}.

We will now show that $\mf(G_m|G)$ tends to 0 with growing $|X|$,
while $\mf(f|G)$ remains bounded away from 0.
This will establish that
provided $G$, a maximum at $x_1$ is more likely than a maximum at $x_m$.
Provided $G$, the probability of a maximum at $x_1$ is always
above $2^{-c_f}$, since
$\mf(\max \text{ at } x_1|G)\geq \cmf{f}{G}\geq \mf(f) \geq 2^{-c_f}$.
A maximum at $x_m$, on the other hand, is less likely since only functions
in $G_m$ can have a maximum there:
\begin{align}
\mf(\max \text{ at }x_m|G)
	&=\mf(G_m)/\mf(G)\nonumber\\
	&\leq\mf(G_m)/ 2^{-c_f}\nonumber\\
        &=2^{c_f}\cdot c_{\m_\cont}\!\!\!\sum_{g\in G_m}\!\!\!2^{-\fK{g}}
        \label{eq:fl:fsteq}\\
\intertext{Using the lower bound on the complexity from Lemma \ref{le:incfunc}, \eqref{eq:fl:fsteq} is bounded by}
	&\leq 2^{c_f} \cdot c_{\m_\cont}\!\!\!\sum_{g\in  G_m} 2^{-\lb{|X|+c}}\nonumber\\
	&=2^{c_f}\cdot c_{\m_\cont} \abs{G_m}\cdot2^{-\lb{|X|+c}}\label{eq:fl:secineq}
\intertext{and since the cardinality of $G_1$ is less than $\abs{\range}^{k+1}$, \eqref{eq:fl:secineq} is bounded by}
	&\leq 2^{c_f}\cdot c_{\m_\cont}\cdot |\range|^{k+1}\cdot 2^{-\lb{|X|+c}}\nonumber\\
	&=\frac{2^{c_f+c}\cdot c_{\m_\cont} \cdot \abs{\range}^{k+1}}{|X|}\label{eq:fl:final}
\end{align}
the last equality by elementary simplification.

As $c_{\m_\cont}$ is bounded above by a constant $c_\m$ for all $X,Y$,
\eqref{eq:fl:final} goes to 0 with growing search space (and fixed
$k$ and \range). This shows that for large enough search spaces, $x_1$
is more likely to host a maximum than $x_m$.

Now all that remains is to use this to create two algorithms that perform
differently under \mptm. Let $a$ start by searching $Q$ in order.
If the perceived function points are consistent with $f$ (i.e., the event
$G$ is verified), then $a$ proceeds at $x_1$ and then at $x_m$,
whereafter $a$ searches the remaining points $\sspace-D_{k}-\{x_1\}$ in order.
If the trace is not consistent with $f$, then $a$ directly proceeds to
search all remaining points in order.
Define $b$ the same way, with the only exception that after $Q$
it searches $x_m$ before $x_1$ in case the trace is consistent with $f$.

This way, $a$ and $b$ will perform the same except when encountering a
function in $G$, in which case $a$ will have a strictly better chance
of finding the maximum at step $\abs{Q}+1$. If neither $a$ nor $b$
finds a maximum at step $\abs{Q}+1$, then neither $x_1$ nor $x_m$ is a
maximum, so neither $a$ nor $b$ will find a maximum at step $\abs{Q}+2$ either.
Finally, on step $\abs{Q}+3$
and onwards their behaviour will again be identical, and therefore also
their \mptm performance.
So $a$ has a strictly better chance at step $\abs{Q}+1$ and $a$ and $b$'s
performance is identical on all other steps and in all other situations.
This shows that there is a (possibly small) free lunch for \mptm on $\m$
for sufficiently large search spaces.
\end{proof}

\end{document}